\documentclass[12pt, a4paper, twosided, reqno]{amsart}
\usepackage{enumerate}
\newtheorem{thm}{Theorem}
\newtheorem{lem}{Lemma}
\newtheorem{cor}{Corollary}
\newtheorem{remark}{Remark}
\newtheorem{defn}{Definition}
\newtheorem{example}{Example}
\theoremstyle{definition}
\author[ M. Saini]{ Manisha Saini}

\thanks {The research work of the author is supported by research fellowship from University Grants Commission (UGC), New Delhi.}
\title[Order and Hyper-Order]{Order and Hyper-order of Solutions of Second Order Linear Differential Equations}

\begin{document}
\maketitle

\begin{abstract}
\noindent We have discussed the problem of finding the condition on coefficients of $f''+A(z)f'+B(z)f=0, \quad B(z)(\not \equiv 0)$ so that all non-trivial solutions are of infinite order. The hyper-order of these non-trivial solutions of infinite order is also found when $\lambda(A)<\rho(A)$ and $B(z)$ is a transcendental entire function satisfying some conditions.
 \textbf{Address:} Department of Mathematics, University of Delhi, New Delhi-110007, India.\\
\textbf{Email:} msaini@maths.du.ac.in, sainimanisha210@gmail.com

\end{abstract}
\section{Introduction}
The study of growth of solutions of complex differential equation starts with Wittich's work in Wittich \cite{witt}. For the fundamental results of complex differential equations we have consulted Hille \cite{helle} and Laine \cite{lainebook}. The Nevanlinna's value distribution theory plays a crucial role  in investigation of complex differential equations.  For the notion of value distribution theory we have consulted the standard reference Yang \cite{yang}.

For an entire function $f(z)$ the order of growth is defined as:
$$ \rho(f) =  \limsup_{r \rightarrow \infty} \frac{\log^+ \log^+ M(r, f)}{\log r}=   \limsup_{r \rightarrow \infty} \frac{\log^+T(r, f)}{\log r}$$
where $M(r,f)=\max \{\ |f(z)|: |z|=r \}\ $ is the maximum modulus of the function $f(z)$ over the circle $|z|=r$ and $T(r,f)$ is the Nevanlinna characteristic function of the function $f(z)$.

In this paper we investigate the growth of solutions $f (\not \equiv 0)$ of  the second order linear differential equation 
\begin{equation}\label{sde}
f''+A(z)f'+B(z)f=0
\end{equation}
where the coefficients $A(z)$ and $B(z)(\not \equiv 0)$ are entire functions. 
It is known that all solutions of the equation (\ref{sde}) are entire functions Laine  \cite{lainebook}. The necessary and sufficient condition that all solutions of the equation (\ref{sde}) are of finite order is that the coefficients $A(z)$ and $B(z)$ are polynomials Laine \cite{lainebook}. It is easy to conclude that if any of the coefficients is a transcendental entire function then almost all solutions are of infinite order. However, there is a necessary condition for equation (\ref{sde}) to have a solution of finite order:

\begin{thm}\cite{finitegg}\label{finitgg}
Suppose that $f(z)$ be a finite order solution of the equation (\ref{sde}) then $T(r,B)\leq T(r,A)+O(1)$.
\end{thm}
This implies that if equation (\ref{sde}) possess a solution of finite order then $\rho(B)\leq \rho(A)$. Therefore, if $\rho(A)<\rho(B)$ then all non-trivial solutions $f(z)$ of the equation (\ref{sde}) are of infinite order. It is well known that above condition is not sufficient, for example: $f''+e^{-z}f'-n^2f=0$ has all non-trivial solutions of infinite order. \\
Therefore, it is interesting to find conditions on $A(z)$ and $B(z)$ so that all solutions $f(\not \equiv 0)$ are of infinite order. Many results have been given in this context. Gundersen \cite{finitegg} and Hellerstein et. al. \cite{heller} proved
\begin{thm}\label{ggthm}
 Let $f(\not \equiv 0)$ be a solution of the equation (\ref{sde}) with the coefficients satisfying 
 \begin{enumerate}
 \item $\rho(B)<\rho(A)\leq \frac{1}{2}$ or
 \item $A(z)$ is a transcendental entire functions with $\rho(A)=0$ and $B(z)$ is a polynomial.
 \end{enumerate}
 then $\rho(f)=\infty.$
 \end{thm}
Further, Frei \cite{frei}, Ozawa \cite{ozawa}, Amemiya and Ozawa \cite{ame}, Gundersen \cite{ggpol} and Langley \cite{lang} proved that all non-trivial solutions are of infinite order for the differential equation 
$$f''+Ce^{-z}f'+B(z)f=0$$
for any nonzero constant $C$ and for any nonconstant polynomial $B(z)$. J. Heittokangas, J. R. Long, L. Shi, X. Wu,  P. C. Wu, X. B. Wu, and Zhang gave conditions on the coefficients $A(z)$ and $B(z)$ so that all solutions $f(\not \equiv 0)$ are of infinite order. Their results can be found in [\cite{jlong}, \cite{extremal}, \cite{wu}, \cite{zhe}]. In Kumar and Saini \cite{dsm} we gave conditions on coefficients  and proved the following theorem:
\begin{thm}\label{prethm}
 Suppose $A(z)$ be an entire function with $\lambda(A)<\rho(A)$ and 
 \begin{enumerate}
 \item $B(z)$ be a transcendental entire function with $\rho(B)\neq \rho(A)$ or
 \item $B(z)$ be an entire function having Fabry gap
 \end{enumerate}
 then all non-trivail solutions of the equation (\ref{sde}) are of infinite order. 
 \end{thm}
\begin{defn}
An entire function $f(z)=\sum_{n=0}^{\infty}a_{\lambda_n}z^{\lambda_n}$ has Fabry gap if the sequence $(\lambda_n)$ satisfies 
$$ \frac{\lambda_n}{n} \rightarrow \infty $$
as $ n \rightarrow \infty$. An entire function $f(z)$ with Fabry gap satisfies $\rho(f)>0$ Hayman and Rossi \cite{hayman}.
\end{defn}

The concept of hyper-order were used to further investigate the growth of infinite order solutions of complex differential equations. In this context, K. H. Kwon \cite{kwo} proved that: 
\begin{thm}
Suppose $P(z)=a_nz^n+\ldots a_0$ and $Q(z)=b_nz^n+\ldots b_0$ be non-constant polynomials of degree $n$  such that either $\arg{a_n}\neq\arg{b_n}$ or $a_n=cb_n$ $(0<c<1)$, $h_1(z)$ and $h_0(z)$ be entire functions satisfying $\rho(h_i)<n$, $i=0,1$. Then every non-trivial solutions $f(z)$ of 
\begin{equation}\label{kwoeq}
f''+h_1e^{P(z)}f'+h_0e^{Q(z)}f=0, \quad Q(z)\not \equiv 0
\end{equation}
are of infinite order with $\rho_2(f)\geq n$.    
\end{thm} 
For an entire function $f(z)$ the hyper-order is defined in the follwoing manner:
$$ \rho_2(f)= \limsup_{r \rightarrow \infty} \frac{\log^+ \log^+ \log^+ M(r, f)}{\log r}=   \limsup_{r \rightarrow \infty} \frac{\log^+ \log^+T(r, f)}{\log r}$$
C. Zongxuan \cite{pz}, investigated the differential equation (\ref{kwoeq}) for some special cases and proved the following theorem:
\begin{thm}
Let $b\neq-1$ be any complex constant, $h(z)$ be a non-zero polynomial. Then every solution $f(\not \equiv 0)$ of the equation
\begin{equation}
f''+e^{-z}f'+h(z)e^{bz}f=0
\end{equation}
has infinite order and $\rho_2(f)=1.$

\end{thm}
K. H. Kwon \cite{kwon} found the lower bound for the hyper-order of all solutions $f(\not \equiv 0)$ in the following theorem:
 \begin{thm}\cite{kwon}\label{kwonthm}
Suppose that  $A(z)$ and $B(z)$ be entire functions such that (i) $\rho(A)<\rho(B)$ or (ii) $\rho(B)<\rho(A)<\frac{1}{2}$ then $$\rho_2(f)\geq \max\{\ \rho(A),\rho(B)\}\ $$
for all solutions $f(\not \equiv 0)$ of the equation (\ref{sde}).
\end{thm}

Since the growth of an entire function with infinite order can be measured by its hyper-order. Therefore motivated from above theorems we have calculated the hyper-order of the  non-trivial solutions of the equation (\ref{sde}) in Theorem [\ref{prethm}] in the following theorem:
 \begin{thm}\label{hypthm}
 Let $A(z)$ and $B(z)$ be entire functions of finite order satisfying the hypothesis of the Theorem $[\ref{prethm}]$ then all non-trivial solutions $f(z)$ of the equation (\ref{sde}) have  $ \rho_2(f)=\max \{\ \rho(A),\rho(B) \}\ $.
 \end{thm}
 
In present work, our aim is to give conditions on $B(z)$ so that when $\rho(A)=\rho(B)$ then also the conclusion of Theorem [\ref{prethm}] and Theorem [\ref{hypthm}] holds true. In this regard, we have proved few results. 

\begin{thm}\label{lowthm}
 Suppose that $A(z)$ and $B(z)$ be transcendental entire functions satisfying $\lambda(A)<\rho(A)$ and $\mu(B)\neq \rho(A)$ then all non-trivial solutions $f(z)$ of the equations satisfies $\rho(f)=\infty$.
 \end{thm}
\begin{cor}\label{corlo}
The conclusion of the above theorem holds true if $\mu(A)\neq \mu(B)$. 
\end{cor}
For an entire function $f(z)$ the lower order of growth is defined as follows:
$$ \mu(f)= \liminf_{r \rightarrow \infty} \frac{\log^+ \log^+ M(r, f)}{\log r}=   \liminf_{r \rightarrow \infty} \frac{\log^+T(r, f)}{\log r}$$
In Theorem [\ref{lowthm}] and Corollary [\ref{corlo}],  the order of the coefficients $A(z)$ and $B(z)$ may be equal. The lower order of an entire function may be quite different from its order, for example there exists entire function $f$ with $\mu(f)=0$ and $\rho(f)>0$  or $\rho(f)=\infty$ for example see Goldberg and Ostroviskii \cite{aa} (page no. 238).

The theorem below presents the hyper-order of solutions of the differential equation satisfying the conditions of the Theorem [\ref{lowthm}].
\begin{thm}\label{lowhyp}
Suppose that $A(z)$ be an entire fucntion with finite order and $B(z)$ be a transcendental entire function with finite lower order satisfying the hypothesis of Theorem [\ref{lowthm}] then
$$\rho_2(f)=\max\{\ \rho(A),\mu(B)\}\ $$
for all non-constant solutions $f(z)$ of the equation (\ref{sde}).
\end{thm}
\begin{cor}
Suppose that $A(z)$ and $B(z)$ be entire functions of finite lower order such that $\lambda(A)<\rho(A)$ and $\mu(A)\neq \mu(B)$ then 
$$\rho_2(f)=\max\{\ \mu(A), \mu(B)\}\ .$$
for all non-constant solutions $f(z)$ of the equation (\ref{sde}).
\end{cor}
\begin{thm}\label{bordi}
Suppose that $A(z)$ be an entire function with $\lambda(A)<\rho(A)$ and $B(z)$ be an entire function extremal to Yang's inequality such that no Borel direction of $B(z)$ coincides with any of the critical rays of $A(z)$. Then all non-trivial solutions $f(z)$ of the equation (\ref{sde}) satisfies $\rho(f)=\infty.$
\end{thm}
Here is an illustrative example for above theorem:
\begin{example}
The differential equation
\begin{equation*}
f''+e^{\iota z}f'+e^zf=0
\end{equation*}
has all non-trivial solutions of infinite order by Theorem [\ref{bordi}].
\end{example}
\begin{thm}\label{bordihy}
Suppose that $A(z)$ be an entire function with finite order and $B(z)$ be an entire function extremal to Yang's inequality such that hypothesis of the Theorem [\ref{bordi}] satisfied then
 $$\rho_2(f)=\max\{\ \rho(A),\rho(B)\}\ $$
for all non-trivial solutions $f(z)$ of the equation (\ref{sde}).
\end{thm}
The follwoing theorem is motivated from Theorem [1.6] of  Wu et.al.  \cite{wu} where coefficients of equation (\ref{sde}) satisfies $\rho(A)\neq\rho(B)$.
\begin{thm}\label{denth1}
Suppose that $A(z)$ be an entire function with $\lambda(A)<\rho(A)$ and $B(z)$ be an entire function extremal to Denjoy's conjecture then all non-trivial solutions $f$ of the equation (\ref{sde}) satisfies
$$ \rho(f)=\infty.$$
\end{thm}
\begin{thm}\label{denth2}
Suppose that $A(z)$ and $B(z)$ be entire function of finite order satisfying the hypothesis of the above theorem then all non-trivial solutions $f$ of the equation (\ref{sde}) satisfies 
$$ \rho_2(f)=\max\{\ \rho(A),\rho(B)\}\ . $$
\end{thm}
Definitions of Borel directions, Denjoy's conjecture, functions extremal to Yang's inequality and extremal to Denjoy's conjecture are given in the next section.
 \section{Preliminary Results}
To make this paper self contained we mention all results which we are going to use and some additional results we have proved.

For a set $F\subset [0,\infty)$, the Lebesgue linear measure of $F$ is defined as $m(F)=\int_{F}dt$ and for a set $G\subset[1,\infty)$, the logarithmic measure of $G$ is defined as $m_1(G)=\int_{G}\frac{1}{t}dt$. For set $G\subset[0,\infty)$, the upper and lower logarithmic densities are defined, respectively, as follows:
	$$\overline{\log dens}(G)=\limsup_{r\rightarrow \infty}\frac{m_1(G\cap[1,r])}{\log{r}} $$
	
	$$\underline{\log dens}(G)=\liminf_{r\rightarrow \infty}\frac{m_1(G\cap[1,r])}{\log{r}}.$$
 Next lemma is due to Gundersen \cite{log gg} which provide the estimates for transcendental meromorphic function.

\begin{lem}\label{gunlem}
Let $f(z)$ be a transcendental meromorphic function and let $\Gamma= \{\ (k_1,j_1), (k_2,j_2), \ldots ,(k_m,j_m) \}\ $ denote finite set of distinct pairs of integers that satisfy $ k_i > j_ i \geq 0$  for $i=1,2, \ldots,m$. Let $\alpha >1$ and $\epsilon>0$ be given real constants. Then the following three statements holds:
\begin{enumerate}[(i)]
\item  there exists a set $E_1 \subset[0,2\pi)$ that has linear measure zero and there exists a constant $c>0$ that depends only on $\alpha$ and $\Gamma$ such that if $\psi_0 \in [0,2\pi)\setminus E_1, $ then there is a constant $R_0=R(\psi_0)>0$  so that for all $z$ satisfying $\arg z =\psi_0$ and $|z| \geq R_0$, and for all $(k,j)\in \Gamma$, we have
\begin{equation} \label{ggguneq}
\left| \frac{f^{(k)}(z)}{f^{(j)}(z)}\right|\leq c \left( \frac{T(\alpha r,f)}{r} \log^{\alpha}{r} \log{T(\alpha r,f)} \right)^{(k-j)}
\end{equation}

If $f(z)$ is of finite order then $f(z)$ satisfies:
\begin{equation} \label{guneq1}
\left|\frac{f^{(k)}(z)}{f^{(j)}(z)}\right| \leq |z|^{(k-j)(\rho(f)-1+\epsilon)}
\end{equation} 
 for all $z$ satisfying $\arg z =\psi_0\notin E_1$ and $|z| \geq R_0$ and for all $(k,j)\in \Gamma$
\item there exists a set $E_2\subset (1,\infty)$ that has finite logarithmic measure and there exists a constant $c>0$ that depends only on $\alpha$ and $\Gamma$ such taht for all $z$ satisfying $|z|=r\notin E_2\cup[0,1]$ and for all $(k,j)\in \Gamma$, inequality $(\ref{ggguneq})$ holds.\\

If $f(z)$ is of finite order then $f(z)$ satisfies inequality $(\ref{guneq1})$,  for all $z$ satisfying $|z| \not \in E_2\cup[0,1]$ and for all $(k,j)\in \Gamma$.
\item \label{gunlem3} there exists a set $E_3\subset [0,\infty)$ that has finite linear measure and there exists a constant $c>0$ that depends only on $\alpha$ and $\Gamma$ such that for all $z$ satisfying $|z|=r\notin E_3$ and for all $(k,j)\in \Gamma$ we have
\begin{equation}
\left| \frac{f^{(k)}(z)}{f^{(j)}(z)}\right| \leq c \left(T(\alpha r,f)r^{\epsilon}\log{T(\alpha r,f)}\right)^{(k-j)}
\end{equation}
If $f(z)$ is of finite order then 
\begin{equation}
\left| \frac{f^{(k)}(z)}{f^{(j)}(z)}\right| \leq |z|^{(k-j)(\rho(f)+\epsilon)}
\end{equation}
 for all $z$ satisfying $|z|\not \in E_3$ and for all $(k,j)\in \Gamma$.
\end{enumerate}
\end{lem}
 Wang \cite{wang} has proved the following result using Phragm$\acute{e}$n-Lindel$\ddot{o}$f theorem.
\begin{lem}\label{entlem}
Let $A(z)$ be an entire funtion such that $\rho(A)\in (0,\infty)$ then there exists sector $\Omega(\alpha,\beta)$ where $\alpha<\beta$ and $\beta-\alpha \geq \frac{\pi}{\rho(A)}$ such that
$$\limsup_{r\rightarrow \infty}\frac{\log \log |A(re^{\iota \theta })|}{\log{r}}=\rho(A)$$
for all $\theta\in(\alpha,\beta).$
\end{lem}
For the statement of our next lemma we need to introduce the notion of \emph{critical rays}:
\begin{defn}
Let $P(z)=a_{n}z^n+a_{n-1}z^{n-1}+\ldots +a_0$, $a_n\neq0$ and $\delta(P,\theta)=\Re(a_ne^{\iota n \theta})$. 
A ray $\gamma = re^{\iota \theta}$  is called \emph{critical ray} of $e^{P(z)}$ if $\delta(P,\theta)=0.$
\end{defn}
 The rays $\arg{z}=\theta$ such that $\delta(P,\theta)=0$ divides the complex plane into $2n$ sectors of equal length $\frac{\pi}{n}$. Also $\delta(P,\theta)>0$ and $\delta(P,\theta)<0$ in the alternative sectors. Suppose that $0\leq\phi_1<\theta_1<\phi_2<\theta_2<\ldots<\phi_n<\theta_n<\phi_{n+1}=\phi_1+2\pi$ be $2n$ critical rays of $e^{P(z)}$ satisfying $\delta(P,\theta)>0$ for $\phi_i<\theta<\theta_i$ and $\delta(P,\theta)<0$ for $\theta_i<\theta<\phi_{i+1}$ where $i=1,2,3, \ldots, n$. Now we fix some notations:
$$ E^+ = \{ \theta \in [0,2\pi]: \delta(P,\theta)\geq 0\}$$ 
$$E^- = \{ \theta \in [0,2\pi]: \delta(P,\theta)\leq 0 \}.$$
\\
Let $\alpha$, $\beta$ and $r_1>0$, $r_2>0$ be such that $\alpha<\beta$ and $r_1<r_2$ then 
$$\Omega(\alpha,\beta)= \{z\in \mathbb{C}: \alpha<\arg z <\beta \}$$
$$\Omega(\alpha,\beta; r_1,r_2)=\{z\in \mathbb{C}:\alpha <\arg{z}<\beta, r_1<|z|<r_2\}.$$

 We state following lemma which is due to Bank et.al. \cite{banklang} and is useful for estimating an entire function $A(z)$ satisfying $\lambda(A)<\rho(A)$.

\begin{lem}\label{implem}
Let $A(z)=v(z)e^{P(z)}$ be an entire function with $\lambda(A)<\rho(A)$, where $P(z)$ is a non-constant polynomial of degree $n$ and $v(z)$ is an entire function. Then for every $\epsilon>0$ there exists $E \subset [0,2\pi)$ of linear measure zero such that

\begin{enumerate}[(i)]

\item for $ \theta  \in E^+\setminus E$ there exists $ R>1 $ such that
\begin{equation}\label{eqA1}
|A(re^{\iota \theta})| \geq \exp{ \left( (1-\epsilon) \delta(P,\theta)r^n \right)}
\end{equation}
for $r>R.$

\item for $\theta \in E^-\setminus E$ there exists $R>1$ such that 
\begin{equation}\label{eq2le}
|A(re^{\iota \theta})| \leq \exp \left( (1-\epsilon)\delta(P,\theta) r^n \right) 
\end{equation}
for $r>R.$
\end{enumerate}
\end{lem}

\begin{lem}\cite{kwon} \label{kwonlem}
Let $f(z)$ be a non-constant entire function. Then there exists a real number $R>0$ such that for all $r\geq R$ we have
\begin{equation}\label{kwoneq}
\left| \frac{f(z)}{f'(z)} \right| \leq r
\end{equation}
where $|z|=r$.
\end{lem}
The lemma below give property of an entire function with Fabry gap and can be found in Long \cite{jlongfab} and Wu and Zheng \cite{zhe}. 
\begin{lem}\label{fablem}
Let $g(z)=\sum_{n=0}^{\infty} a_{\lambda_n}z^{\lambda_n}$ be an entire function of finite order with Fabry gap, and $h(z)$ be an entire function with $\rho(h)=\sigma \in (0,\infty)$. Then for any given $\epsilon\in (0,\sigma)$, there exists a set $H\subset (1,+\infty)$ satisfying $ \overline{\log dens}(H) \geq \xi $, where $\xi\in (0,1)$ is a constant  such that for all $|z| =r \in H$, one has
$$ \log M(r,h) > r^{\sigma-\epsilon}, \quad \log m(r,g) > (1-\xi)\log M(r,g),$$

where $M(r,h)=\max \{\ |h(z)|: |z|=r\}\ $, $m(r,g)=\min \{\ |g(z)|: |z|=r\}\ $ and $M(r,g)= \max \{\ |g(z)|: |z|=r\}\ $.
\end{lem}
The following remark follows from the above lemma immediately.
\begin{remark} \label{fabremark}
Suppsoe that $g(z)=\sum_{n=0}^{\infty} a_{\lambda_n}z^{\lambda_n}$ be an entire function of order $\sigma \in (0,\infty)$ with Fabry gap then for any given $\epsilon >0, \quad (0<2\epsilon <\sigma)$, there exists a set $H\subset (1,+\infty)$ satisfying $\overline{\log dens}(H) \geq \xi$, where $\xi \in (0,1)$ is a constant such that for all $|z| =r \in H$ , one has
$$ |g(z)|> M(r,g)^{(1-\xi)}> \exp{\left((1-\xi) r^{\sigma-\epsilon}\right)}>\exp{\left(r^{\sigma-2\epsilon}\right)}.$$
\end{remark}
\begin{lem}\cite{chen}\label{indexlem}
Let $f(z)$ be an entire function of infinite order then
$$\rho_2(f)=\limsup_{r\rightarrow\infty}\frac{\log{\log{v(r,f)}}}{\log{r}}$$
where $v(r,f)$ is the central index of the function $f(z)$.
\end{lem}
C. Zongxuan \cite{pz} provides the upper bound for the hyper-order of solutions $f(z)$ of the equation (\ref{sde}).
\begin{thm}\label{wuthm}
Suppose that $A(z)$ and $B(z)$ are entire functions of finite order. Then 
\begin{equation*}
\rho_2(f)\leq\max\{\ \rho(A), \rho(B)\}\
\end{equation*}
for all solutions $f(z)$ of the equation (\ref{sde}).
\end{thm}

The following result is from Wiman-Valiron theory and we use this result to prove our next lemma which is motivated from Theorem [\ref{wuthm}].
\begin{thm}\label{wim}\cite{lainebook}
Let $g$ be a transcendental entire function, let $0<\delta<\frac{1}{4}$ and $z$ be such that $|z|=r$ and 
$$|g(z)|>M(r,g)v(r,g)^{-\frac{1}{4}+\delta}$$
holds. Then there exists a set $F\subset \mathbb{R}_+$ of finite logarithmic measure such that 
$$ g^{(m)}(z)=\left( \frac{v(r,g)}{z} \right)^m(1+o(1))g(z)$$
holds for all $m\geq0$ and for all $r\notin F$, where $v(r,g)$ is the central index of the function $g(z)$. 
\end{thm}

\begin{lem}\label{ourlem}
Let us suppose that $A(z)$ and $B(z)$ be entire functions such that $\mu(A)$ and $\mu(B)$ are finite then 
$$\rho_2(f)\leq \max\{\ \mu(A),\mu(B)\}\ $$
for all solutions $f$ of the equation (\ref{sde}).
\end{lem}
\begin{proof}
Suppose $\max\{\ \mu(A),\mu(B)\}\ = \rho$. Thus for $\epsilon>0$ we have 
\begin{equation}\label{eq1}
|A(re^{\iota \theta}|\leq \exp{r^{\rho+\epsilon}}
\end{equation}
and 
\begin{equation}\label{eq2}
|B(re^{\iota \theta}|\leq \exp{r^{\rho+\epsilon}}
\end{equation}
for sufficiently large $r$.
From Theorem [\ref{wim}], we choose $z$ satisfying $|z|=r$ and $|f(z)|=M(r,f)$ then there exists a set $F\subset \mathbb{R_+}$ having finite logarithmic measure such that
\begin{equation}\label{wimeq}
\frac{f^{(m)}(z)}{f(z)}=\left( \frac{v(r,f)}{z} \right)^m(1+o(1))
\end{equation}
for $m=1,2$ and for all $|z|=r\notin F$, where $v(r,f)$ is the central index of the function $f(z)$. Thus using equation (\ref{sde}), (\ref{eq1}), (\ref{eq2}) and (\ref{wimeq}) we get
\begin{equation}
\left( \frac{v(r,f)}{z} \right)^2|(1+o(1))|\leq \exp{ \left(r^{\rho+\epsilon}\right)} \left( \frac{v(r,f)}{z} \right)|(1+o(1))|+\exp{\left(r^{\rho+\epsilon}\right)}
\end{equation}
for all $|z|=r\notin F$, from here we get 
\begin{equation}
\limsup_{r\rightarrow \infty} \frac{\log \log{v(r,f)}}{\log{r}}\leq \rho+\epsilon.
\end{equation}
Since $\epsilon>0$ chosen is arbitrary we get $\rho_2(f)\leq \rho.$
\end{proof}
 Theorem [\ref{kwonthm}] motivated us to prove following result:
\begin{lem}
Suppose that $A(z)$ and $B(z)$ be entire function such that $\mu(A)<\mu(B)$ then 
$$\rho_2(f)\geq \mu(B)$$ 

for all non-trivial solutions $f$ of the equation (\ref{sde}).
\end{lem}
\begin{proof}
Let $\mu<\alpha<\beta<\mu(B)$ where $\alpha$ and $\beta$ are two real numbers.Let $f$ be a non-trivial solutions of the equation (\ref{sde}). For given $\epsilon>0$, from part (ii) of Lemma [\ref{gunlem}], there exists $E_2\subset (1,\infty)$ with finite logarithmic measure and a constant $c>0$ such that
\begin{equation}\label{hypmu}
\left|\frac{f^{(k)}(z)}{f(z)}\right| \leq c (T(2r,f)) ^{2k}\quad k=1,2
\end{equation}
for all $z$ satisfying $|z|=r\notin E_2\cup[0,1]$. From equations (\ref{sde}) and (\ref{hypmu})  we get
\begin{align*}
\exp{r^{\beta}}&\leq|B(re^{\iota\theta})|\leq\left|\frac{f''(re^{\iota \theta})}{f(\iota \theta)}\right|+|A(re^{\iota \theta})|\left|\frac{f'(re^{\iota \theta})}{f(\iota \theta)}\right|\\
& \leq cT(2r,f)^4(1+\exp{r^{\alpha}})
\end{align*}
for all $r\notin E_2\cup[0,1]$. Since $\alpha<\beta$ this implies
$$\limsup_{r\rightarrow \infty}\frac{\log \log T(r,f)}{\log r}\geq \beta$$
as $\beta\leq \mu(B)$ is arbitrary this implies
$$\rho_2(f)\geq \mu(B).$$
\end{proof}
\begin{lem}\label{loword}\cite{wu1}
Suppose $B(z)$ be an entire function with $\mu(B)\in[0,1)$. Then for every $\alpha \in(\mu(B),1)$, there exists a set $E_4\subset[0,\infty)$ such that 

$$\overline{\log dens}(E_4)\geq 1-\frac{\mu(B)}{\alpha} \mbox{ and } m(r)>M(r)\cos{\pi\alpha}$$

 for all $r\in E_4$, where $m(r)=\inf_{|z|=r}\log |B(z)|$ and   \\
\qquad $M(r)=\sup_{|z|=r} \log |B(z)|$.
\end{lem}
The above lemma is also true for an entire function $B(z)$ with $\rho(B)<\frac{1}{2}.$ We can get next lemma easily using Lemma [\ref{loword}].
\begin{lem}\label{lowordB}\cite{wu1}
If $B(z)$ be an entire function with $\mu(B)\in(0,\frac{1}{2})$. Then for any $\epsilon>0$ there exists $(r_n)\rightarrow \infty$ such that 
$$|B(r_ne^{\iota \theta})|>\exp{r_n^{\mu(B)-\epsilon}}$$
for all $\theta \in [0,2\pi)$.
\end{lem}

\begin{lem}\label{jlole}\cite{wu}
Let $B(z)$ be an entire function with $\mu(B)\in[\frac{1}{2},\infty)$. Then there exists a sector $\Omega(\alpha, \beta),$ $\beta-\alpha\geq \frac{\pi}{\mu(B)}$, such that 
$$\limsup_{r\rightarrow \infty}\frac{\log \log |B(re^{\iota \theta|})}{\log r}\geq \mu(B)$$
for all $\theta \in \Omega(\alpha,\beta)$, where $0\leq\alpha<\beta\leq2 \pi.$
\end{lem}
Next we give definition of Borel direction and illustrate it with an example:
\begin{defn}\cite{yang}
For a meromorphic function $f(z)$ of order $\rho(f)\in(0,\infty)$ in the finite plane, the ray $\arg{z}=\theta_0$ is called Borel direction of $f$ of order $\rho(f)$ if for any $\epsilon>0$, the equality
 
$$\limsup_{r\rightarrow \infty}\frac{\log{n(\Omega(\theta_0-\epsilon,\theta_0+\epsilon,r), f=a)}}{\log{r}}=\rho(f)$$

holds for every complex number $a$, with atmost two possible exceptions, where $n(\Omega(\theta_0-\epsilon,\theta_0+\epsilon, r),f=a)$ denotes the number of zeros, counting with the multiplicities, of the function $f(z)-a$ in the region $\Omega(\theta_0-\epsilon,\theta_0+\epsilon,r)$ 
\end{defn}
\begin{example}
The entire function $f(z)=e^{z}$ has two Borel directions namely $\frac{\pi}{2}$ and $-\frac{\pi}{2}$.
\end{example}
Next result gives relation between the number of the deficient values and number of Borel directions.
\begin{thm}\label{extr}
Let $f(z)$ be an entire function of order $\rho(f)\in(0,\infty)$. If $p$ is the number of its finite deficient values and $q$ is the number of its Borel directions, then $p\leq \frac{q}{2}$.
\end{thm}
 If equality $p=\frac{q}{2}$ holds in Theorem [\ref{extr}] then function $f(z)$ is called extremal to Yang's inequality. For example:
\begin{example}
Consider the entire function $f(z)=\int_{0}^{z}e^{-t^n}dt$ of order $n$ has $n$ number of finite deficient values equal to 
$$a_k=e^{\frac{i2\pi k}{n}}
\int_{0}^{\infty}e^{-t^n}dt, k=0,1,2,\ldots, n-1$$ 
and $2n$ Borel directions equal to 
$$\Phi_i=\frac{(2i-1)\pi}{2n}, i=0,1,2,\ldots, 2n-1.$$ Since $p=\frac{q}{2}$ therefore this fuction is extremal to Yang's inequality.
\end{example}

Next suppose that $B(z)$ be an entire function extremal to Yang's inequality and let $\arg{z}=\Phi_i$, $i=1,2,3,\ldots,q$  denote the Borel directions of the function $B(z)$ such that $0\leq \Phi_1<\Phi_2<\ldots<\Phi_q<\Phi_{q+1}=\Phi_1+2\pi$. The following lemma is due to \cite{sjwu}:
\begin{lem}\label{defic}
Suppose that $B(z)$ be an entire function extremal to Yang's inequality and $b_i$, $i=1,2,3,\ldots, \frac{q}{2}$ be the deficient values of $B(z)$. Then for each $b_i$, $i=1,2,3,\ldots,\frac{q}{2}$ there exists a corresponding sector $\Omega(\Phi_i,\Phi_{i+1})$ such that for every $\epsilon>0$ 
\begin{equation}\label{defic}
\log{\frac{1}{|B(z)-b_i|}}>C(\Phi_i,\Phi_{i+1}, \epsilon, \delta(b_i,B))T(|z|,B)
\end{equation} 
holds for all $z\in \Omega(\Phi_i+\epsilon,\Phi_{i+1}-\epsilon, r,\infty)$, $C(\Phi_i,\Phi_{i+1}, \epsilon, \delta(b_i,B))$ is a positive constant depending on $\Phi_i, \Phi_{i+1}, \epsilon$ and $\delta(b_i,B)$.
\end{lem}
\begin{lem}\label{expinc}\cite{extremal}
Suppose that $B(z)$ be an entire function extremal to Yang's inequality and there exists $\arg{z}=\theta$ with $\Phi_i<\theta<\Phi_{i+1},$ $1\leq j\leq q$ such that 
\begin{equation}\label{expinc}
\limsup_{r\rightarrow \infty}\frac{\log{\log{|B(re^{\iota \theta})|}}}{\log{r}}=\rho(B).
\end{equation}
Then $\Phi_{i+1}-\Phi_i=\frac{\pi}{\rho(B)}$.
\end{lem}
\begin{lem}\label{expincr}\cite{yangch}
Suppose that $f(z)$ be an entire function with $0<\rho(f)<\infty$ and $\Omega(\psi_1,\psi_2)$ be a sector with $\psi_2-\psi_1<\frac{\pi}{\rho(f)}$. If there exists a Borel direction $\arg{z}=\Phi$ in $\Omega(\psi_1,\psi_2)$ then there exists atleast one of the rays $\arg{z}=\psi_i$, $i=1$ or $2$ such that 
\begin{equation}\label{expincr}
\limsup_{r\rightarrow \infty}\frac{\log{\log{|f(re^{\iota \psi_i})|}}}{\log{r}}=\rho(f)
\end{equation}
\end{lem}

Here we give a conjecture due to Denjoy \cite{den} which gives a relation between the order of an entire function and its finte asymptotic values:
\\

 {\bf{Denjoy's Conjecture:}}  Suppose $g(z)$ be an entire function of finite order and $g$ has $p$ distinct finite asymptotic values then $p\leq 2\rho(g)$.

An entire function is said to be extremal to Denjoy's conjecture if equality holds in above inequality. For example:
\begin{example}[\cite{zh}, page no. 210]
Let 
$$g(z)=\int_{0}^{z} \frac{sin{t^p}}{t^p}dt$$
where $p\in\mathbb{N}$. Then $\rho(g)=p$ and $g(z)$ has $2p$ distinct finite asymptotic values, namely
$$a_j=e^{\frac{j\pi \iota}{p}}\int_{0}^{\infty}\frac{sin{r^p}}{r^p}dr$$
for $j=1,2,\ldots 2p$.
\end{example}
Following lemma gives the property of an entire function extremal to Denjoy's conjecture.
\begin{lem}\label{denlem}\cite{zh}
Let $g(z)$ be an entire function extremal to Denjoy's conjecture then for any $\theta\in(0,2\pi)$ either $\arg{z}=\theta$ is a Borel direction of $g(z)$ or there exists a constant $\sigma\in(0,\frac{\pi}{4})$ such that 
\begin{equation}
\lim_{{|z|\rightarrow \infty}_{ z \in (\Omega(\theta-\sigma,\theta+\sigma)\setminus E_5)}} \frac{\log \log |g(z)|}{\log{|z|}}
=\rho(g)
\end{equation}
where $E_5\subset \Omega(\theta-\sigma,\theta+\sigma)$ such that 
$$\lim_{r\rightarrow \infty}m(\Omega(\theta-\sigma,\theta+\sigma; r, \infty)\cap E_5)=0$$
\end{lem}
\section{Proof of Theorem [\ref{hypthm}]}

\begin{proof}

\begin{enumerate}
\item We know that all solutions $f(\not \equiv 0)$ of the equation (\ref{sde}) are of infinite order, when $\rho(B)\neq \rho(A)$ by Theorem [\ref{prethm}]. Then from part (iii) of Lemma [\ref{gunlem}] for $\epsilon>0$, there exists a set $E_3\subset [0,\infty)$ that has finite linear measure such that for all $z$ satisfying $|z|=r\notin E_3$ we have
\begin{equation}\label{eqfi}
\left| \frac{f''(z)}{f'(z)}\right| \leq c r \left[T(2r,f)\right]^2
\end{equation}
where $c>0$ is a constant.\\
If $\rho(A)<\rho(B)$ then from Theorem [\ref{kwonthm}] and Theorem [\ref{wuthm}] we get that $\rho_2(f)= \max\{\ \rho(A),\rho(B) \}\ $.\\

 If $\rho(B)<\rho(A)=n, n\in \mathbb{N}$ then we can choose $\beta$ such that $\rho(B)<\beta<\rho(A)$. Now choose $\theta \in E^+\setminus E$ and $(r_m)\not \subset E_3$ such that equations (\ref{eqA1}), (\ref{kwoneq}) and (\ref{eqfi}) are satisfied for  $z_m=r_me^{(\iota \theta)}$. Using equation (\ref{sde}), (\ref{eqA1}), (\ref{kwoneq}) and (\ref{eqfi})  for  $z_m=r_me^{(\iota \theta)}$ we have
\begin{align*}
\qquad \exp{\{\ (1-\epsilon)\delta(P,\theta)r_m^n\}\ }&\leq |A(r_me^{\iota \theta})|\\
&\leq \left|\frac{f''(r_me^{\iota \theta})}{f'(r_me^{\iota \theta)}}\right|+|B(r_me^{\iota \theta})|\left| \frac{f(r_me^{\iota \theta})}{f'(r_me^{\iota \theta})} \right| \\
&\leq cr_m\left[T(2r_m,f)\right]^2+\exp{\left(r_m^{\beta}\right)}r_m
\end{align*} 
since $\beta < n$ this implies that 
\begin{equation}\label{eqthm1}
\limsup_{r\rightarrow \infty}\frac{\log \log T(r,f)}{\log{r}} \geq \rho(A)
\end{equation}
then from Theorem [\ref{wuthm}] and equation (\ref{eqthm1}) we have
$$ \rho_2(f)=\max\{\ \rho(A),\rho(B)\}\ $$

\item It has been proved that all non-trivial solutions $f(z)$ of the equation (\ref{sde}), with $A(z)$ and $B(z)$ satisfying the hypothesis of the theorem, are of infinite order. Also if $\rho(A)\neq \rho(B)$ then from above part [1],  
$$\rho_2(f)=\{\ \rho(A), \rho(B)\}\ $$
Now let $\rho(A)=\rho(B)=n, n\in\mathbb{N}$. Using Lemma [\ref{gunlem}], for $\epsilon>0$, there exists $E_3\subset [0,\infty)$ with finite linear measure such that for all $z$ satisfying $|z|=r\not \in E_3$ we have
\begin{equation}\label{eqhyp}
\left| \frac{f^{(k)}(z)}{f(z)}\right|\leq cr T(2r,f)^{2k}
\end{equation}
where $c>0$ is a constant and $k\in \mathbb{N}$. Also from Lemma [\ref{fablem}], for $\epsilon>0$, there exist $H\subset (1,\infty)$ satisfying $\overline{\log dens}(H)\geq 0$ such that for all $|z|=r\in H$ we have 
\begin{equation}\label{eqfab}
|B(z)|>\exp{\left(r^{n-\epsilon}\right)}
\end{equation}
Next choose $\theta \in E^-\setminus E$, $\delta(P,\theta)<0$ and $r_m \in H \setminus E_3$, from equations (\ref{sde}), (\ref{eq2le}), (\ref{eqhyp}) and (\ref{eqfab}) we have
\begin{align*}
\qquad \quad \exp{\left(r_m^{n-\epsilon}\right)}&< |B(r_me^{\iota \theta})|\leq \left| \frac{f''(r_me^{\iota \theta})}{f(r_me^{\iota \theta})}\right| +|A(r_me^{\iota \theta})| \left| \frac{f'(r_me^{\iota \theta})}{f(r_me^{\iota \theta})}\right| \\
& \leq cr_m T(2r_m,f)^4\\
&+ \exp{\{\ (1-\epsilon)\delta(P,\theta)r^{n}\}\ }cr_mT(2r_m,f)^2 \\
&\leq cr_m T(2r_m,f)^4(1+o(1)).
\end{align*}
Thus we conclude that 
\begin{equation}\label{hypor}
\limsup_{r\rightarrow \infty} \frac{\log \log T(r,f)}{\log r} \geq n.
\end{equation}
Using Theorem [\ref{wuthm}] and equation (\ref{hypor}) we get 
$$\rho_2(f)=\{\ \rho(A), \rho(B) \}\ . $$
\end{enumerate}
\end{proof}
\section{Proof of Theorem [\ref{lowthm}]}

\begin{proof}
If $\rho(A)=\infty$ then result follows from equation (\ref{sde}). Assume that $\rho(A)<\infty$.

 If $\rho(A)<\mu(B)$ then result follows from Theorem [\ref{finitgg}]. Let us suppose that $\mu(B)<\rho(A)$ and $f(z)$ be a non-trivial solution of the equation (\ref{sde}) with finite order. Then using part (i) of Lemma [\ref{gunlem}], for each $\epsilon>0,$ there exists a set $E_1 \subset[0,2\pi)$ that has linear measure zero, such that if $\psi_0 \in [0,2\pi)\setminus E_1, $ then there is a constant $R_0=R_0(\psi_0)>0$ and
\begin{equation} \label{guneq}
\left|\frac{f^{(k)}(z)}{f(z)}\right| \leq |z|^{2\rho(f)}, \quad k=1,2
\end{equation}
 for all $z$ satisfying $\arg z =\psi_0$ and $|z| \geq R_0$.
Since $\lambda(A)<\rho(A)$ therefore $A(z)=v(z)e^{P(z)}$, where $P(z)$ is a non-constant polynomial of degree $n$ and $v(z)$ is an entire function such that $\rho(v)=\lambda(A)<\rho(A)$. Then using Lemma [\ref{implem}], there exists $E\subset [0,2\pi)$ with linear measure zero such that for $\theta \in E^-\setminus (E\cup E_1)$ there exists $R_1>1$ such that 
\begin{equation}\label{eqA}
 |A(re^{\iota \theta})|\leq \exp \left( (1-\epsilon) \delta(P,\theta)r^n \right) 
\end{equation} 
 for $r>R_1.$\\
 We have following three cases on lower order of $B(z)$:\\

{\bf{Case 1.}} when $0<\mu(B)<\frac{1}{2}$ then from Lemma [\ref{lowordB}],  there exists $(r_n)\rightarrow \infty$ such that 
\begin{equation}\label{eqB}
|B(re^{\iota \theta})|>\exp{\left(r^{\mu(B)-\epsilon} \right)}
\end{equation}
for all $\theta \in [0,2\pi)$ and $r>R_3,$ $r\in (r_n)$.\\
Using equation (\ref{sde}), (\ref{guneq}), (\ref{eqA}) and (\ref{eqB}) we have
\begin{align*}
\exp{\left(r^{\mu(B)-\epsilon}\right)}&<|B(z)|\leq \frac{|f''(z)|}{|f(z)|}+|A(z)|\frac{|f'(z)|}{|f(z)|} \\
&\leq r^{2\rho(f)}\{\ 1+\exp \left( (1-\epsilon) \delta(P,\theta)r^n \right)\}\ \\
&= r^{2\rho(f)} \{\ 1+o(1) \}\
\end{align*}
for all $\theta \in E^-\setminus (E \cup E_1)$ and $r>R, r\in(r_n)$. This will conduct a contradication for sufficiently large $r$.
\\
Thus all non-trivial solutions are of infinite order in this case.\\

{\bf{Case 2.}} Now if $\mu(B)\geq \frac{1}{2}$ then by Lemma [\ref{jlole}] we have that there exists a sector $\Omega(\alpha,\beta)$, $0\leq \alpha <\beta \leq 2\pi$, $\beta-\alpha\geq \frac{\pi}{\mu(B)}$ such that 
\begin{equation}\label{eqB1}
\limsup_{r\rightarrow \infty} \frac{\log \log |B(re^{\iota \theta}|}{\log r}\geq \mu(B)
\end{equation}
for all $\theta \in \Omega(\alpha, \beta)$.\\
Since $\mu(B)<\rho(A)$ therefore there exists $\Omega(\alpha',\beta') \subset \Omega(\alpha,\beta)$ such that for all $\phi \in \Omega(\alpha', \beta')$ we have 
\begin{equation}\label{eqA2}
|A(re^{\iota \phi}|\leq \exp{\left( (1-\epsilon)\delta(P,\theta)r^n\right)}
\end{equation}
for all $r>R.$ From equation (\ref{eqB1}) we get
\begin{equation}\label{eqB2}
\exp{\left( r^{\mu(B)-\epsilon}\right)}\leq |B(re^{\iota \phi})|
\end{equation}
for $\phi \in \Omega(\alpha',\beta')$ and $r>R.$ As done in above case, using equation (\ref{sde}), (\ref{guneq}), (\ref{eqA2}) and (\ref{eqB2}) we get contradiction for sufficiently large $r$.\\

{\bf{Case 3.}} If $\mu(B)=0$ then from Lemma [\ref{loword}] for $\alpha \in (0,1)$, there exists a set $E_4\subset [0,\infty)$ with $\overline{\log dens}(E_4)=1$ such that 
$$m(r)>M(r)\cos{\pi \alpha}$$
 where $m(r)=\inf_{|z|=r}\log |B(z)|$ and $M(r)=\sup_{|z|=r} \log|B(z)|$. Then 
\begin{equation}\label{eqB3}
\log{|B(re^{\iota \theta})|}>\log{M(r,B)}\frac{1}{\sqrt{2}}
\end{equation}
for all $\theta \in [0,2\pi)$ and $r\in E_4.$ Now using equation (\ref{sde}), (\ref{guneq}), (\ref{eqA2}) and (\ref{eqB3}) we get
\begin{align*}
M(r,B)^{\frac{1}{\sqrt{2}}}<|B(re^{\iota \theta})|\leq r^{2\rho(f)}\{\ 1+\exp{(1-\epsilon)\delta(P,\theta)r^n} \}\
\end{align*}
for $\theta \not \in E\cup E_1$, $\delta(P,\theta)<0$ and $r>R$, $r\in E_4$. This implies that
 $$\liminf_{r\rightarrow \infty} \frac{\log M(r,B)}{\log r}<\infty$$
which is not so as $B(z)$ is an transcendental entire function. Thus non-trivial solution $f$ with finite order of the equation (\ref{sde}) can not exist in this case also.
Therefore all non-trivial solutions of the equation (\ref{sde}) are of infinite order.
\end{proof}
\section{Proof of Theorem [\ref{lowhyp}]}
\begin{proof}
We know that under the hypothesis of the theorem, all non-trivial solutions $f(z)$ of the equation (\ref{sde}) are of infinite order. It follows from Lemma [\ref{gunlem}] that for $\epsilon>0$, there exists a set $E_3\subset [0,\infty)$ with finte linear measure such that
\begin{equation}\label{eqf1}
\left|\frac{f''(z)}{f'(z)}\right|\leq cr[T(2r,f)]^2
\end{equation}
for all z satisfying $|z|=r\notin E_3$ when $c>0$ is a constant.\\

If $\rho(A)<\mu(B)$ then from Theorem [\ref{kwonthm}] and Lemma [\ref{ourlem}] we get that $\rho_2(f)=\max\{\ \rho(A),\mu(B)\}\ $, for all non-trivial solutions $f(z)$ of the equation (\ref{sde}).\\

If $\mu(B)<\rho(A)$. It is easy to choose $\eta$ such that $\mu(B)<\eta <\rho(A)$. From Lemma [\ref{implem}], we have
\begin{equation}\label{entA}
\exp{\{\ (1-\epsilon)\delta(P,\theta)r^n \}\ }\leq |A(re^{\iota \theta}|
\end{equation}
for all $\theta\notin E$, $\delta(P,\theta)>0$ and for sufficiently large $r$.\\
Also 
\begin{equation}\label{entB}
|B(re^{\iota \theta})|\leq\exp{r^{\eta}}
\end{equation}
for sufficiently large $r$ and for all $\theta\in[0,2\pi)$.\\
Thus from equations (\ref{sde}), (\ref{kwoneq}), (\ref{eqf1}),  (\ref{entA}) and (\ref{entB}) we have
\begin{align*}
\exp{\{\ (1-\epsilon)\delta(P,\theta)r^n\}\ }&\leq |A(re^{\iota \theta}| \\
&\leq \left|\frac{f''(re^{\iota \theta})}{f'(re^{\iota \theta})}\right|+|B(re^{\iota \theta})|\left|\frac{f(re^{\iota \theta})}{f'(re^{\iota \theta})}\right| \\
&\leq cr[T(2r,f)]^2+\exp{\left(r^{\eta}\right)}r \\
& \leq dr\exp{\left(r^{\eta}\right)}[T(2r,f)]^2
\end{align*}
for all $\theta \notin E$, $\delta(P,\theta)>0$ and for sufficiently large $r$. Since $\eta<\rho(A)=n$ we have
\begin{equation}
\exp{\{\ (1-\epsilon)\delta(P,\theta)\}\ }\exp{\{\ (1-o(1))r^n\}\ }\leq dr[T(2r,f)]^2
\end{equation}
for sufficiently large $r$. Thus 
$$\limsup_{r\rightarrow\infty}\frac{\log{\log{T(r,f)}}}{\log{r}}\geq n.$$
Now using $ \rho_2(f)\geq\max\{\ \rho(A),\mu(B) \}\ $and Lemma [\ref{ourlem}] we get the desired result.
\end{proof}
\section{Proof of Theorem [\ref{bordi}]}
\begin{proof}
If we consider the coefficients $A(z)$ and $B(z)$ such that $\rho(A)\neq\rho(B)$ then result follows from Theorem [\ref{prethm}]. Therefore, it is sufficient to consider $\rho(A)=\rho(B)=n$ for some $n\in \mathbb{N}$. Suppose there exists a non-trivial solution $f(z)$ of the equation (\ref{sde}) of finite order. Then using part (ii) of Lemma [\ref{gunlem}], for each $\epsilon>0,$ there exists a set $E_2 \subset(1,\infty)$ that has finite logarithmic measure, such that
\begin{equation} \label{gguneq}
\left|\frac{f^{(k)}(z)}{f(z)}\right| \leq |z|^{2\rho(f)}, \quad k=1,2
\end{equation}
 for all $z$ satisfying $|z|\not \in E_2\cup [0,1]$.

 Since $B(z)$ is extremal to Yang's inequality therefore there exists sectors $\Omega_i(\Phi_i,\Phi_{i+1})$, $i=1,2,3,\ldots, q$ such that in alternative sectors either equation (\ref{defic}) or equation (\ref{expinc}) holds for the function $B(z)$. Let $\Omega_1(\Phi_1,\Phi_2)$, $\Omega_3(\Phi_3,\Phi_4)$, $\ldots, \Omega_{2q-1}(\Phi_{2q-1},\Phi_{2q})$ being the sectors such that 
\begin{equation}\label{eqBde}
\log{\frac{1}{|B(z)-b_i|}}>C  T(|z|,B)
\end{equation} 
holds for all $z\in \Omega_i(\Phi_i+\epsilon,\Phi_{i+1}-\epsilon, r,\infty)$, $C=C(\Phi_i,\Phi_{i+1}, \epsilon, \delta(b_i,B))$, where $\delta(b_i,B)$ is used for deficiency function of $B(z)$, is a positive constant depending on $\Phi_i, \Phi_{i+1}, \epsilon$ and $\delta(b_i,B)$, where $i=1,3,\ldots, 2q-1$.

Also, let $\Omega_2(\Phi_2,\Phi_3),$ $\Omega_4(\Phi_4,\Phi_6)$, $\ldots, \Omega_{2q}(\Phi_{2q},\Phi_{2q+1})$ are the sectors for which there exists $re^{\iota \theta_{2i}}\in \Omega_{2i}(\Phi_{2i},\Phi_{2i+1})$ such that 
\begin{equation}\label{expincre}
\limsup_{r\rightarrow \infty}\frac{\log{\log{|B(re^{\iota \theta_{2i}})|}}}{\log{r}}=n
\end{equation}
holds and $\Phi_{2i+1}-\Phi_{2i}=\frac{\pi}{n}$ where $i=1,2,\ldots, q$.
  
Now we have the following cases to be discussed:\\

{\bf{Case 1.}} let us suppose that there is a Borel direction $\Phi$ of $B(z)$ such that $\theta_i<\Phi<\phi_{i+1}$ for any $i=1,2,3, \ldots,n$ then we can easily choose $\psi_1$ and $\psi_2$ such that $\theta_i<\psi_1<\Phi<\psi_2<\phi_{i+1}$. It is evident from Lemma [\ref{expincr}] that without loss of generality we can choose $\psi_2$ and we get
\begin{equation}\label{exp}
\limsup_{r\rightarrow \infty} \frac{\log{\log{|B(re^{\iota \psi_2})|}}}{\log{r}}=n.
\end{equation}
Thus, for $r\notin E_2\cup[0,1]$ from equation (\ref{sde}), (\ref{eq2le}), (\ref{gguneq}) and (\ref{exp}) we have 
\begin{align*}
\exp{\{\ r^{n-\epsilon}\}\ }&\leq |B(re^{\iota \psi_2})|\leq \left| \frac{f''(re^{\iota \psi_2})}{f(re^{\iota \psi_2})}\right| +|A(re^{\iota \psi_2})|\left|\frac{f'(re^{\iota \psi_2})}{f(re^{\iota \psi_2})}\right| \\
& \leq r^{2\rho(f)}\{\ 1+\exp{\{\ 1-\epsilon)\delta(P,\psi_2)r^n\}\ }\}\
\end{align*}
which is a contradiction for sufficiently large $r$.\\

{\bf{Case 2.}} Now suppose that there is no Borel direction of $B(z)$ contained in $(\theta_i,\phi_{i+1})$ for any $i=1,2,\ldots, n$. In this case $(\theta_i,\phi_{i+1})$ will be contained inside $\Omega_{2j-1}(\Phi_{2j-1},\Phi_{2j})$, for any $j=1,2,3,\ldots, q$.

Therefore for $r\notin E_2\cup[0,1]$ and $\theta\in E^+\cap \Omega_{2j-1}(\Phi_{2j-1},\Phi_{2j})$, from equations (\ref{sde}), (\ref{eqA1}), (\ref{kwoneq}), (\ref{gguneq}) and (\ref{eqBde}) we get
\begin{align*}
\qquad \exp{\{\ (1-\epsilon)\delta(P,\theta)r^n \}\ }&\leq |A(re^{\iota \theta})| \\
&\leq \left|\frac{f''(re^{\iota \theta})}{f'(re^{\iota \theta})}\right|+|B(re^{\iota \theta})|\left|\frac{f(re^{\iota \theta})}{f'(re^{\iota \theta})}\right| \\
	&\leq r^{2\rho(f)}+r\{\ \exp{\{\ -CT(r,B)\}\ }+|b_{2j-1}|\}\ \\
&\leq r^{2\rho(f)}(1+|b_{2j-1}|+o(1))
\end{align*} 
which provides a contradition for sufficiently large $r$.

Thus all non-trivial solutions $f$ of the equation (\ref{sde}) are of infinite order.
\end{proof}
\section{ Proof of the Theorem [\ref{bordihy}]}
\begin{proof}
Since all non-trivial solution of the equation (\ref{sde}) under hypothesis are of infinite order. Therefore it follows from part (iii) of Lemma [\ref{gunlem}] that for $\epsilon>0$, there exists a set $E_3\subset [0,\infty)$ having finite linear measure such that for all $z$ satisfying $|z|=r\notin E_3$ we have
\begin{equation}\label{eqf}
\left| \frac{f^{(k)}(z)}{f^{(j)}(z)}\right| \leq c r \left[T(2r,f)\right]^{2(k-j)}
\end{equation}
where $c>0$ is a constant and $k\in \mathbb{N}$.\\
If $\rho(A)\neq\rho(B)$ then from part (\ref{parti}) of Theorem [\ref{hypthm}] we have 
$$\rho_2(f)= \max\{\ \rho(A),\rho(B)\}\ . $$
We consider $\rho(A)=\rho(B) =n$ where $n\in \mathbb{N}$. Now we have two cases to deal with

{\bf{Case 1}}. let us suppose that there is a Borel direction $\Phi$ of $B(z)$ such that $\theta_i<\Phi<\phi_{i+1}$ for any $i=1,2,3, \ldots,n$.Thus, for $r\notin E_3$ from equation (\ref{sde}), (\ref{eq2le}), (\ref{exp}) and  (\ref{eqf})we have 
\begin{align*}
\exp{\{\ r^{n-\epsilon}\}\ }&\leq |B(re^{\iota \psi_2})|\leq \left| \frac{f''(re^{\iota \psi_2})}{f(re^{\iota \psi_2})}\right| +|A(re^{\iota \psi_2})|\left|\frac{f'(re^{\iota \psi_2})}{f(re^{\iota \psi_2})}\right| \\
& \leq  c r \left[T(2r,f)\right]^4+ \exp{\{\ (1-\epsilon)\delta(P,\psi_2)r^n\}\ }c r \left[T(2r,f)\right]^2\\
&\leq  c r \left[T(2r,f)\right]^4(1+o(1))
\end{align*} 
for sufficiently large $r$. Thus 
\begin{equation}\label{bordieq}
\limsup_{r\rightarrow \infty}\frac{\log{\log{T(r,f)}}}{\log{r}}\geq n
\end{equation}
Then it can be seen easily from equation (\ref{bordieq}) and  Theorem [\ref{wuthm}] that
 $$\rho_2(f)=n$$
for all non-trivial solutions $f$ of the equation (\ref{sde}).\\

{\bf{Case 2.}} Now suppose that there is no Borel direction of $B(z)$ contained in $(\theta_i,\phi_{i+1})$ for any $i=1,2,\ldots, n$. Therefore for $r\notin E_3$ and $\theta\in E^+\cap \Omega_{2j-1}(\Phi_{2j-1},\Phi_{2j})$, from equations (\ref{sde}), (\ref{eqA1}), (\ref{kwoneq}), (\ref{eqBde}) and  (\ref{eqf}) we get
\begin{align*}
\qquad \exp{\{\ (1-\epsilon)\delta(P,\theta)r^n\}\ }&\leq |A(re^{\iota \theta})| \\
&\leq \left|\frac{f''(re^{\iota \theta})}{f'(re^{\iota \theta})}\right|+|B(re^{\iota \theta})|\left|\frac{f(re^{\iota \theta})}{f'(re^{\iota \theta})}\right| \\
&\leq c r \left[T(2r,f)\right]^2+r \{\ \exp{\{\ -CT(r,B)\}\ }\\
&+|b_{2j-1}|\}\  \\
&\leq dr \left[T(2r,f)\right]^2(1+|b_{2j-1}|+o(1))
\end{align*}
for sufficiently large $r$ and for $d>0$ is a constant. Thus
\begin{equation}\label{bordieq1}
\limsup_{r\rightarrow \infty}\frac{\log{\log{T(r,f)}}}{\log{r}}\geq n
\end{equation}
It follows from equation (\ref{bordieq1}) and Theorem [\ref{wuthm}] that
$$\rho_2(f)=n$$
for all non-trivial solution $f$ of the equation (\ref{sde}).

\end{proof}
\section{Proof of the Theorem [\ref{denth1}]}
\begin{proof}
When $\rho(A)\neq \rho(B) $ then the result holds true from Theorem [prethm]. Assume that $\rho(A)=\rho(B)=n, n\in \mathbb{N}$ and there exists a non-trivial solution $f$ of the equation (\ref{sde}) of finite order. Then by part (i) of Lemma [\ref{gunlem}],  for given $\epsilon>0$ equation (\ref{guneq})holds true for $z$ satisfying $|z|>R$ and $\arg{z}\in E_1$. We will discuss following cases:\\
{\bf{Case 1.}} suppose that the ray $\arg{z}=\Phi$ is a Borel direction of $B(z)$ where $\theta_i<\Phi<\phi_{i+1}$ for some $i=1,2,\ldots, n$ then the conclusion holds in similar manner as in Case 1. of Theorem [\ref{bordi}].
\\
{\bf{Case 2.}} Suppose that $\arg{z}=\theta$ is not a Borel direction of $B(z)$ for any $\theta\in(\theta_i,\phi_{i+1})$ for all $i=1,2,\ldots, n$ then choose $\arg{z}=\theta \in (\theta_i,\phi_{i+1})$ for some $i=1,2,\ldots, n$. Then by Lemma [\ref{denlem}] there exists $\sigma\in(0,\frac{\pi}{4})$ such that 
\begin{equation*}
\lim_{|z|\rightarrow \infty, z \in (\Omega(\theta-\sigma,\theta+\sigma)\setminus E_5)} \frac{\log \log |B(z)|}{\log{|z|}}
=\rho(B)
\end{equation*}
Thus  \begin{equation}\label{deneq}
\exp{\{\ r^{\rho(B)-\epsilon}\}\ }\leq |B(z)|
\end{equation}
for all $z$ satisfying $|z|=r\rightarrow \infty$ and $z\in (\Omega(\theta-\sigma,\theta+\sigma)\setminus E_5)\cap (\theta_i,\phi_{i+1})\setminus S$, where $S=\{\ z\in \mathbb{C}: \arg(z)\in E_1\}\ $. Now from equations (\ref{sde}), (\ref{eq2le}), (\ref{guneq}) and (\ref{deneq}) we get a contradiction for sufficiently large $r$.\\

Thus all non-trivial solutions of the equation (\ref{sde}) are of infinite order.
\end{proof}
\section{Proof of the Theorem [\ref{denth2}]}
\begin{proof}
We need to consider that $\rho(A)=\rho(B)=n, n\in\mathbb{N}$. We again discuss two cases:
\\
{\bf{Case 1.}} suppose that the ray $\arg{z}=\Phi$ is a Borel direction of $B(z)$ where $\theta_i<\Phi<\phi_{i+1}$ for some $i=1,2,\ldots, n$ then the conclusion holds in similar manner as in Case 1. of Theorem [\ref{bordi}].\\
{\bf{Case 2.}} Suppose that $\arg{z}=\theta$ is not a Borel direction of $B(z)$ for any $\theta\in(\theta_i,\phi_{i+1})$ for all $i=1,2,\ldots, n$ then choose $\arg{z}=\theta \in (\theta_i,\phi_{i+1})$ for some $i=1,2,\ldots, n$. Then from equations (\ref{sde}), (\ref{ggguneq}), (\ref{eq2le}) and (\ref{deneq}) we have 
\begin{equation}\label{deneq1}
\limsup_{r\rightarrow \infty}\frac{\log \log T(r,f)}{\log{r}}\geq n
\end{equation}
From Theorem [wuthm] and equation (\ref{deneq1}) we get the desired result.
\end{proof}
\section{Extention to Higher Order}
This section involves linear differential equation of the form:
\begin{equation}\label{sde1}
f^{(k)}+A_{k-1}(z)f^{(k-1)}+A_{k-2}f^{(k-2)}+\ldots+A_{0}f=0
\end{equation}
where $A_{k-1}, A_{k-2},\ldots, A_0$ are entire functions, therefore all solutions of the equation (\ref{sde1}) are entire function \cite{lainebook}. Also, all solutions of the equation (\ref{sde1}) are of finite order if and only if all the coefficients $A_{k-1}, A_{k-2},\ldots, A_0$ are polynomials. Thus if any of the coefficients $A_{k-1}, A_{k-2},\ldots, A_0$ is a transcendental entire functions then there will exists a non-trivial solution of infinite order. In this section we will discuss the conditions on coefficients $A_{k-1}, A_{k-2}, \ldots, A_0$ so that all non-trivial solutions of the equation (\ref{sde1}) are of infinite order. 

For this purpose we will extend our previous results in the following manner:
\begin{thm}\label{exth1}
Suppose that there exists $A_i(z)$ such that $\lambda(A_i)<\rho(A_i)$ and $A_0(z)$ be a transcendental entire function satisfying $\mu(A_0)\neq \rho(A_i)$ and $\rho(A_j)<\mu(A_0)$ for all $j=1,2,\ldots, k-1$ and $j\neq i$. Then all non-trivial solutions of the equation $(\ref{sde1})$ are of infinite order. Moreover, for these solutions $f$ we have 
$$\rho_2(f)\geq \mu(A_0) .$$
\end{thm}
\begin{cor}
The conclusion of the above theorem also holds if $\mu(A_0)\neq \mu(A_i)$ and $\rho(A_j)<\mu(A_0)$ for all $j=1,2,\ldots, k-1$ and $j\neq i$.
\end{cor}
\begin{cor}
The conclusion of the theorem also holds if $\mu(A_0)\neq \mu (A_i)$ and $\mu(A_j)<\mu(A_0)$ for all $j=1,2,\ldots, k-1$ and $j\neq i$.
\end{cor}
\begin{thm}\label{exth2}
 Suppose that $A_1(z)$ be an entire function with $\lambda(A_1) < \rho(A_1)$ and $A_0(z)$ be an entire function extremal to Yang’s inequality such that no Borel direction of $A_0(z)$ coincides with any of the critical rays of $A_1(z)$ and $\rho(A_j)<\rho(A_0)$, where $j=2,\ldots, k-1$. Then all non-trivial solutions of the equation $(\ref{sde1})$  satisfies
$$\rho(f)=\infty \quad \mbox{and} \quad \rho_2(f)\geq \rho(A_0) .$$
\end{thm}
\begin{cor}
The conclusion of the above theorem also holds true if $\mu(A_j)<\rho(A_0)$.
\end{cor}
\begin{thm}\label{exth3}
Suppose there exist $A_i(z)$ such that $\lambda(A_i)<\rho(A_i)$ and $A_0(z)$ be an entire function extremal to Denjoy's conjecture and  $\rho(A_j)<\rho(A_0), j=1,2,\ldots, k-1, j\neq i$. Then all non-trivial solutions $f$ of the equation $(\ref{sde1})$ satisfies 
$$\rho(f)=\infty \quad \mbox{and} \quad \rho_2(f)\geq \rho(A_0) .$$
\end{thm}
\begin{cor}
The conclusion of the above theorem holds true also if $\mu(A_j)<\rho(A_0), j=1,2,\ldots, k-1, j\neq i$.
\end{cor}
\section{Proof of the Theorem [\ref{exth1}]}
\begin{proof}
Assume that there exists a non-trivial solution $f$ of the equation (\ref{sde1}) with finite order then by part (\ref{gunlem3}) of Lemma [\ref{gunlem}], for given $\epsilon>0$ there exists a set $E_3\subset [0,\infty)$ that has finite linear measure such that
\begin{equation}\label{exeq1}
\left| \frac{f^{(m)}(z)}{f(z)}\right| \leq |z|^{k(\rho(f)+\epsilon)}, \quad m=1,2,3,\ldots, k
\end{equation}
for all $z$ satisfying $|z|\not \in E_3$.\\

We suppose the case when $\rho(A_i)<\mu(A_0)$ then from equations (\ref{sde1}) and (\ref{exeq1}) we get 
\begin{align*}
\left|A_0(z)\right|&\leq \left| \frac{f^{(k)}(z)}{f(z)}\right| +|A_{k-1}(z)|\left|\frac{f^{(k-1)}(z)}{f(z)}\right|+\ldots+|A_1(z)|\left|\frac{f'(z)}{f(z)}\right|\\
&\leq |z|^{k(\rho(f)+\epsilon)}\left[ 1+|A_{k-1}(z)|+\ldots+|A_1(z)|\right]
\end{align*}
for all $z$ satisfying $|z|\not \in E_3$. Which implies that 
$$T(r,A_0)\leq k(\rho(f)+\epsilon)\log{r}+ (k-1)T(r,A_m)+O(1)$$
where $T(r,A_m)=\max\{\ T(r,A_p): p=1,2,\ldots, k-1 \}\ $ and $|z|=r\notin E_3$.This gives us that $\mu(A_0)\leq \mu(A_m), m=1,2, \ldots, k-1$ which is a contradiction. Thus all non-trivial solutions of the equation (\ref{sde1}) are of infinite order. \\
Suppose that $f$ be a nontrial solution of the equation (\ref{sde1}) then by part (ii) of Lemma [\ref{gunlem}], for $\epsilon >0$ there exists a set $E_2\subset (1,\infty)$ that has finite logarithmic measure and there exists a constant $c>0$  such taht for all $z$ satisfying $|z|=r\notin E_2\cup[0,1]$ we have
\begin{equation} \label{exeqf}
\left| \frac{f^{(m)}(z)}{f(z)}\right|\leq c \left( T(2 r,f)\right)^{2k}\quad m=1,2,\ldots, k
\end{equation}
Choose $\max\{\ \rho(A_p): p=1,2,\ldots k-1\}\ <\eta< \mu(A_0) $ then from equations (\ref{sde1}) and (\ref{exeqf}) we get
\begin{align*}
\exp{(\mu(A_0-\epsilon))}&\leq |A_0(z)|\leq \left|\frac{f^{(k)}(z)}{f(z)}\right|+|A_{(k-1)}|\left|\frac{f^{(k-1)}}{f(z)}\right|+\\
&\ldots+|A_1(z)|\left|\frac{f'(z)}{f(z)}\right|\\
&\leq cT(2r,f)^{2k}[1+(k-1)\exp{(r^{\eta})}]
\end{align*}
for all $z$ satisfying $|z|=r\notin E_2$. This will implies that 
$$ \limsup_{r\rightarrow \infty}\frac{\log \log T(r,f)}{\log r}\geq \mu(A_0).$$

Now we consider the case when $\mu(A_0)<\rho(A_i)=n$, where $n\in\mathbb{N}$ and there is a non-trivial solution $f$ of the equation (\ref{sde1}) of finite order then by part (i) of Lemma [\ref{gunlem}], for given $\epsilon>0$ there exists a set $E_1 \subset[0,2\pi)$ that has linear measure zero such that if $\psi_0 \in [0,2\pi)\setminus E_1, $ then there is a constant $R_0=R(\psi_0)>0$  so that for all $z$ satisfying $\arg z =\psi_0$ and $|z| \geq R_0$ we have
\begin{equation} \label{exeq2}
\left|\frac{f^{(m)}(z)}{f(z)}\right| \leq |z|^{k\rho(f)} m=1,2,\ldots ,k
\end{equation}
Also $\rho(A_j)<\mu(A_0)$ for $j=1,2,\ldots,k-1, j\neq i$ then we can choose $\eta>0$ such that 
$$ \max\{\ \rho(A_j), j=1,2,\ldots.k-1, j\neq i\}\ <\eta<\mu(A_0).$$
From above we have that 
\begin{equation}\label{exeq3}
|A_j(z)|\leq \exp{r^{\eta}}, \quad j=1,2,\ldots, k-1, j\neq i
\end{equation}
We have following cases to discuss:
\\
{\bf{Case 1.}} when $0<\mu(A_0)<\frac{1}{2}$ then Lemma [\ref{lowordB}], equations  (\ref{eq2le}), (\ref{sde1}), (\ref{exeq2}) and (\ref{exeq3})  gives
\begin{align*}
\exp{r^{\mu(A_0)-\epsilon}}&\leq  |z|^{k\rho(f)}\left[ 1+\exp{\left((1-\epsilon)\delta(P, \theta)r^n\right)}+(k-2)\exp{r^{\eta}}\right]
\end{align*}
for all $z$ satisfying $|z|=r>R$ and $\arg{z}\in E^-\setminus(E_1\cup E)$. This gives a contradiction for sufficiently large $r$.
\\

{\bf{Case 2.}} Assume that $\mu(A_0)\geq \frac{1}{2}$ then by Lemma [\ref{jlole}], equations (\ref{eq2le}),  (\ref{sde1}),  (\ref{exeq2}) and (\ref{exeq3}) we get a contradiction. 
\\

{\bf{Case 3.}} Suppose that $\mu(A_0)=0$ then using Lemma [\ref{loword}], equations (\ref{eq2le}), (\ref{sde1}),   (\ref{exeq2}) and (\ref{exeq3}) we again get a contradiction. \\

Therefore all non-trivial solutions of the equation (\ref{sde1}) are of infintie order.

If $\mu(A_0)=0$ then $\rho_2(f)\geq 0$ for all non-trivial solutions $f$ of the equation (\ref{sde1}). Therefore we suppose that $\mu(A_0)>0$ then using Lemma [\ref{lowordB}] (or Lemma [\ref{jlole}] ),  equations  (\ref{eq2le}), (\ref{sde1}),  (\ref{exeqf}) and (\ref{exeq3}) we get 
 $$\rho_2(f)\geq\mu(A_0)$$
for all non-trivial solutions $f$ of the equation (\ref{sde1}).
\end{proof}
\section{Proof of the Theorem [\ref{exth2}]}
\begin{proof}
When $\rho(A_0)\neq \rho(A_1)$ then the result follows from \cite{dsm}. Thus we need to consider $\rho(A_0)=\rho(A_1)=n$ where $n\in \mathbb{N}$. Let us suppose that there exists a non-trivial solution $f$ of the equation (\ref{sde1}) of finite order. Then from part (ii) of Lemma [\ref{gunlem}], for $\epsilon>0$ there exists a set $E_2 \subset(1,\infty)$ with finite logarithmic measure such that
\begin{equation} \label{exeq4}
\left|\frac{f^{(m)}(z)}{f^{(p)}(z)}\right| \leq |z|^{(m-p)\rho(f)}, \quad m, p=0, 1,2,\ldots, k, p<m
\end{equation} 
for all $z$ satisfying $|z|=r\notin E_2\cup[0,1]$.\\
Since $\rho(A_j)<\rho(A_0)$ for $j=2,3,\ldots, k-1$ then we choose $\eta>0$ such that 
$$ \max\{\ \rho(A_j):j=2,3, \ldots,k-1\}\ <\eta<\rho(A_0)$$
so that 
\begin{equation}\label{exeq5}
|A_j(z)|\leq \exp{r^{\eta}}
\end{equation}
 for $j=2,3,\ldots, k-1.$ As done earlier in Theorem [\ref{bordi}], we have following two cases to discuss: 
\\
{\bf{Case 1.}} if there exists a Borel direction $\Phi$ of $A_0(z) $ such that $\theta_i<\Phi<\phi_{i+1}$ for $i=1,2,\ldots,n$ then from equations  (\ref{eq2le}), (\ref{exp}), (\ref{sde1}),   (\ref{exeq4}) and (\ref{exeq5}) we have
\begin{align*}
\exp^{(n-\epsilon)}&\leq |A_0(z)|\leq \left|\frac{f^{(k)}(z)}{f(z)}\right|+|A_{(k-1)}(z)|\left|\frac{f^{(k-1)}(z)}{f(z)}\right|+\\
&\ldots+|A_1(z)|\left|\frac{f'(z)}{f(z)}\right|\\
&\leq |z|^{k\rho(f)}[1+(k-2)\exp{(r^{\eta})}+\exp{(1-\epsilon)\delta(P,\psi_2)}]
\end{align*}
for all $z$ satisfying $|z|=r\notin E_2\cup[0,1]$ and $\arg{z}=\psi_2$. This we lead us to a contradiction for large values of $r$. Thus all non-trivial solutions of the equation (\ref{sde1}) are of infinite order.
 From equations (\ref{eq2le}), (\ref{exp}), (\ref{sde1}),   (\ref{exeqf}) and (\ref{exeq5}) we have $\rho_2(f)\geq\rho(A_0)$ for all non-trivial solutions $f$ of the equation (\ref{sde1}).\\
{\bf{Case 2.}} If there does not exists any Borel direction of $A_0(z)$ contained in $(\theta_i,\phi_{i+1})$ for $i=1,2,\ldots, n$ then from equations (\ref{eqA1}), (\ref{kwoneq}), (\ref{eqBde}), (\ref{sde1}), (\ref{exeq4}) and (\ref{exeq5}) we have
\begin{align*}
\exp{((1-\epsilon)\delta(P,\theta)r^n)}&\leq |A_1(z)|\leq \left|\frac{f^{(k)}(z)}{f'(z)}\right|+|A_1(z)|\left|\frac{f^{(k-1)}(z)}{f'(z)}\right|+\\
&\ldots+|A_0(z)|\left|\frac{f(z)}{f'(z)}\right|\\
&\leq |z|^{k\rho(f)}[1+(k-2)\exp(r^{\eta})\\
&+r(\exp{\left(-CT(r,A_0)\right)}+|a_{2j-1}|)]
\end{align*}
for all $|z|=r\notin E_2\cup[0,1]$ and $\arg{z}=\theta \in E^+\cap \Omega_{2j-1}(\Phi_{2j-1}$, where $a_i, i=1,2,\ldots \frac{q}{2}$ are deficient values of $A_0(z)$. Which will provide a contradiction for sufficiently large $r$.\\
Thus all non-trivial solutions of the equation (\ref{sde1}) are of infinite order.
From equations  (\ref{eqA1}), (\ref{kwoneq}), (\ref{eqBde}), (\ref{sde1}),  (\ref{exeqf}) and (\ref{exeq5}) we have $\rho_2(f)\geq\rho(A_0)$ for all non-trivial solutions $f$ of the equation (\ref{sde1}).
\end{proof}
\section{Proof of the Theorem [\ref{exth3}]}
\begin{proof}
If $\rho(A_i)\neq\rho(A_0)$ then result is true from \cite{dsm}. Assume that $\rho(A_i)=\rho(A_0)=n, n\in \mathbb{N}$ and there exists a non-trivial solution $f$ of the equation (\ref{sde1}) of finite order. Then we have following two cases to discuss:
\\
{\bf{Case 1.}} when the ray $\arg{z}=\Phi$ is a Borel direction of $A_0(z)$ where $ \Phi\in(\theta_i,\phi_{i+1})$ for some $i=1,2,\ldots, n$. Choose $\psi_1<\psi_2$ such that $\theta_i<\psi_1<\phi<\psi_2<\phi_{i+1}$ and $\psi_2-\psi_1<\frac{\pi}{\rho(A_i)}=\frac{\pi}{\rho(A_0)}$. Then by Lemma [\ref{expincr}] we can have
\begin{equation}\label{exeq6}
\limsup_{r\rightarrow \infty}\frac{\log{\log{|A_0(re^{\iota \psi_2})|}}}{\log{r}}=\rho(A_0)
\end{equation}
Thus from equations  (\ref{eq2le}), (\ref{sde1}),  (\ref{exeq4}), (\ref{exeq5}) and  (\ref{exeq6}) we get contradiction for sufficiently large $r$.\\
As done in Case 1 of Theorem [\ref{exth2}] we get $\rho_2(f)\geq\rho(A_0)$ for all non-trivail solutions $f$ of the equation  (\ref{sde1}).\\
{\bf{Case 2.}} Suppose that $\arg{z}=\theta$ is not a Borel direction of $A_0(z)$ for any $\theta\in(\theta_i,\phi_{i+1})$ for all $i=1,2,\ldots, n$ then choose $\arg{z}=\theta \in (\theta_i,\phi_{i+1})$ for some $i=1,2,\ldots, n$.  Then Lemma [\ref{denlem}], equations (\ref{guneq1}), (\ref{eq2le}), (\ref{sde1}),  and (\ref{exeq3}) leads us to a contradiction for sufficiently large $r$. \\
Thus all non-trivial solutions of the eqaution  (\ref{sde1}) are of infinite order. Also,
Lemma [\ref{denlem}], equations  (\ref{ggguneq}),  (\ref{eq2le}), (\ref{sde1}), and (\ref{exeq3}) gives that $\rho_2(f)\geq\rho(A_0)$ for all non-trivail solutions $f$ of the equation  (\ref{sde1}).
\end{proof}
{\bf{Acknowledgement:}}  I am thankful to my thesis advisor for his valuable comments and suggestions.
I am also thankful to the Department of Mathematics, Deen Dayal Upadhyaya College (University of Delhi), for providing
the proper research facilities.

\end{document}